\documentclass[11pt]{amsart}
\usepackage{mathptmx, amssymb}

\newtheorem{theorem}{Theorem}[section]
\newtheorem{lemma}[theorem]{Lemma}
\theoremstyle{definition}
\newtheorem{definition}[theorem]{Definition}
\newtheorem{example}[theorem]{Example}
\newtheorem{notation}[theorem]{Notation}
\numberwithin{equation}{theorem}

\def\ge{\geqslant}
\def\le{\leqslant}
\def\bar{\overline}

\def\to{\longrightarrow}
\def\D{\Delta}
\def\fpt{\operatorname{fpt}}
\def\lct{\operatorname{lct}}
\def\init{\operatorname{in}}
\def\modulo{\operatorname{modulo}}

\def\frakm{\mathfrak{m}}

\def\FF{\mathbb{F}}
\def\NN{\mathbb{N}}

\begin{document}
\title{The $F$-pure threshold of a determinantal ideal}

\author{Lance Edward Miller}
\address{Department of Mathematical Sciences, University of Arkansas, Fayetteville, AR~72701, USA} \email{lem016@uark.edu}

\author{Anurag K. Singh}
\address{Department of Mathematics, University of Utah, 155 S 1400 E, Salt Lake City, UT~84112,~USA} \email{singh@math.utah.edu}

\author{Matteo Varbaro}
\address{Dipartimento di Matematica, Universit\`a di Genova, Via Dodecaneso 35, I-16146 Genova, Italy} \email{varbaro@dima.unige.it}

\thanks{L.E.M~was supported by VIGRE grant~0602219, and A.K.S.~by NSF grant DMS~1162585. A.K.S. and M.V. were also supported by NSF grant~0932078000 while in residence at MSRI}

\begin{abstract}
The $F$-pure threshold is a numerical invariant of prime characteristic singularities, that constitutes an analogue of the log canonical thresholds in characteristic zero. We compute the $F$-pure thresholds of determinantal ideals, i.e., of ideals generated by the minors of a generic matrix.
\end{abstract}

\maketitle

\section{Introduction}
Consider the ring of polynomials in a matrix of indeterminates $X$, with coefficients in a field of prime characteristic. We compute the $F$-pure thresholds of determinantal ideals, i.e., of ideals generated by the minors of $X$ of a fixed size.

The notion of \emph{$F$-pure thresholds} is due to Takagi and Watanabe~\cite{TW}, see also Musta\c t\u a, Takagi, and Watanabe~\cite{MTW}. These are positive characteristic invariants of singularities, analogous to log canonical thresholds in characteristic zero. While the definition exists in greater generality---see the above papers---the following is adequate for our purpose:

\begin{definition}
Let $R$ be a polynomial ring over a field of characteristic $p>0$, with the homogeneous maximal ideal denoted by $\frakm$. For a homogeneous proper ideal $I$, and integer $q=p^e$, set
\[
\nu_I(q)=\max\big\{r\in\NN\mid I^r\nsubseteq\frakm^{[q]}\big\}\,,
\]
where $\frakm^{[q]}=(a^q\mid a\in\frakm)$. If $I$ is generated by $N$ elements, it is readily seen that $0\le\nu_I(q)\le N(q-1)$. Moreover, if $f\in I^r\setminus\frakm^{[q]}$, then $f^p\in I^{pr}\setminus\frakm^{[pq]}$. Thus,
\[
\nu_I(pq)\ge p\nu_I(q)\,.
\]
It follows that $\big\{\nu_I(p^e)/p^e\big\}_{e\ge1}$ is a bounded monotone sequence; its limit is the \emph{$F$-pure threshold of $I$}, denoted $\fpt(I)$.
\end{definition}

The $F$-pure threshold is known to be rational in a number of cases, see, for example, \cite{BMS:MMJ, BMS:TAMS, BSTZ, Hara, KLZ:JALG}. The theory of $F$-pure thresholds is motivated by connections to log canonical thresholds; for simplicity, and to conform to the above context, let $I$ be a homogeneous ideal in a polynomial ring over the field of rational numbers. Using $``I\modulo p"$ to denote the corresponding characteristic~$p$ model, one has the inequality
\[
\fpt(I\modulo p)\le\lct(I)\qquad\text{ for all }p\gg0\,,
\]
where $\lct(I)$ denotes the log canonical threshold of $I$. Moreover,
\begin{equation}
\label{eqn:limit}
\lim_{p\to\infty}\fpt(I\modulo p)=\lct(I)\,.
\end{equation}
These follow from work of Hara and Yoshida~\cite{Hara-Yoshida}; see \cite[Theorems~3.3,~3.4]{MTW}.

The $F$-pure thresholds of defining ideals of Calabi-Yau hypersurfaces are computed in~\cite{Bhatt-Singh}. Hern\'andez has computed $F$-pure thresholds for binomial hypersurfaces~\cite{Hernandez1} and for diagonal hypersurfaces~\cite{Hernandez2}. In the present paper, we perform the computation for determinantal ideals:

\begin{theorem}
\label{theorem:main}
Fix positive integers $t\le m\le n$, and let $X$ be an $m\times n$ matrix of indeterminates over a field $\FF$ of prime characteristic. Let $R$ be the polynomial ring~$\FF[X]$, and $I_t$ the ideal generated by the size $t$ minors of $X$.

The $F$-pure threshold of $I_t$ is
\[
\min\left\{\frac{(m-k)(n-k)}{t-k}\mid k=0,\dots,t-1\right\}\,.
\]
\end{theorem}

It follows that the $F$-pure threshold of a determinantal ideal is independent of the characteristic: for each prime characteristic, it agrees with the log canonical threshold of the corresponding characteristic zero determinantal ideal, as computed by Johnson~\cite[Theorem~6.1]{Johnson} or Docampo~\cite[Theorem~5.6]{Docampo} using log resolutions as in Vainsencher~\cite{Vainsencher}. In view of \eqref{eqn:limit}, Theorem~\ref{theorem:main} recovers the calculation of the characteristic zero log canonical threshold.

\section{The computations}

The primary decomposition of powers of determinantal ideals, i.e., of the ideals $I_t^s$, was computed by DeConcini, Eisenbud, and Procesi~\cite{DEP} in the case of characteristic zero, and extended to the case of \emph{non-exceptional} prime characteristic by Bruns and Vetter~\cite[Chapter~10]{BV}. By Bruns~\cite[Theorem~1.3]{Bruns}, the intersection of the primary ideals arising in a primary decomposition of $I_t^s$ in non-exceptional characteristics, yields, in all characteristics, the integral closure $\bar{I_t^s}$. We record this below in the form that is used later in the paper:

\begin{theorem}[Bruns]
\label{theorem:Bruns}
Let $s$ be a positive integer, and let $\delta_1,\dots,\delta_h$ be minors of the matrix $X$. If $h\le s$ and $\sum_i\deg\delta_i=ts$, then
\[
\delta_1\cdots\delta_h\ \in\ \bar{I_t^s}\,.
\]
\end{theorem}

\begin{proof}
By \cite[Theorem~1.3]{Bruns}, the ideal $\bar{I_t^s}$ has a primary decomposition
\[
\bigcap_{j=1}^t I_j^{((t-j+1)s)}\,.
\]
Thus, it suffices to verify that
\[
\delta_1\cdots\delta_h\ \in\ I_j^{((t-j+1)s)}
\]
for each $j$ with $1\le j\le t$. This follows from~\cite[Theorem~10.4]{BV}.
\end{proof}

We will also need:

\begin{lemma}
\label{lemma:min}
Let $k$ be the least integer in the interval $[0,t-1]$ such that
\[
\frac{(m-k)(n-k)}{t-k}\ \le\ \frac{(m-k-1)(n-k-1)}{t-k-1};
\]
interpreting a positive integer divided by zero as infinity, such a $k$ indeed exists. Set
\[
u\ =\ t(m+n-2k)-mn+k^2\,.
\]
Then $t-k-u\ge0$.

Moreover, if $k$ is nonzero, then $t-k+u>0$; if $k=0$, then $t(m+n-1)\le mn$.
\end{lemma}

\begin{proof}
Rearranging the inequality above, we have
\[
t(m+n-2k-1)\ \le\ mn-k^2-k\,,
\]
which gives $t-k-u\ge0$. If $k$ is nonzero, then the minimality of $k$ implies that
\[
t(m+n-2k+1)\ >\ mn-k^2+k\,,
\]
equivalently, that $t-k+u>0$. If $k=0$, the assertion is readily verified.
\end{proof}

\begin{notation}
\label{notation}
Let $X$ be an $m\times n$ matrix of indeterminates. Following the notation in \cite{BV}, for indices
\[
1\le a_1<\dots<a_t\le m\qquad\text{and}\qquad 1\le b_1<\dots<b_t\le n\,,
\]
we set $[a_1,\dots,a_t\mid b_1,\dots,b_t]$ to be the minor
\[
\det\begin{pmatrix}
x_{a_1b_1}&\hdots&x_{a_1b_t}\\
\vdots&&\vdots\\
x_{a_tb_1}&\hdots&x_{a_tb_t}\end{pmatrix}\,.
\]
We use the lexicographical term order on $R=\FF[X]$ with
\[
x_{11}>x_{12}>\dots>x_{1n}>x_{21}>\dots>x_{m1}>\dots>x_{mn}\,;
\]
under this term order, the initial form of the minor displayed above is the product of the entries on the leading diagonal, i.e.,
\[
\init\big([a_1,\dots,a_t\mid b_1,\dots,b_t]\big)\ =\ x_{a_1b_1}x_{a_2b_2}\cdots x_{a_tb_t}\,.
\]

For an integer $k$ with $0\le k\le m$, we set $\D_k$ to be the product of minors:
\begin{multline*}
\prod_{i=1}^{n-m+1}[1,\dots,m\mid i,\dots,i+m-1]\\
\times\prod_{j=2}^{m-k}[j,\dots,m\mid 1,\dots,m-j+1]\cdot[1,\dots,m-j+1\mid n-m+j,\dots,n]\,.
\end{multline*}
If $k\ge1$, we set $\D_k'$ to be
\[
\D_k\cdot[m-k+1,\dots,m\mid 1,\dots,k]\,.
\]
Notice that $\deg\D_k=mn-k^2-k$ and that $\deg\D_k'=mn-k^2$. The element $\D_k$ is a product of $m+n-2k-1$ minors and $\D_k'$ of $m+n-2k$ minors.
\end{notation}

\begin{example}
We include an example to assist with the notation. In the case $m=4$ and $n=5$, the elements $\D_2$ and $\D_2'$ are, respectively, the products of the minors determined by the leading diagonals displayed below:
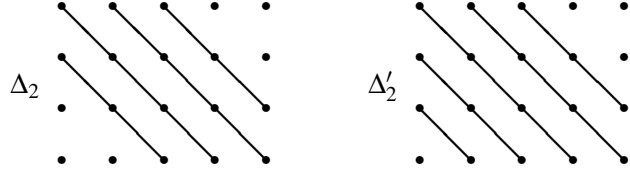
\begin{center}
\setlength{\unitlength}{0.68mm}
\begin{picture}(110,35)
\thicklines
\put(-10,13){$\D_2$}
\put(0,0){\circle*{1.5}}
\put(0,10){\circle*{1.5}}
\put(0,20){\circle*{1.5}}
\put(0,30){\circle*{1.5}}
\put(10,0){\circle*{1.5}}
\put(10,10){\circle*{1.5}}
\put(10,20){\circle*{1.5}}
\put(10,30){\circle*{1.5}}
\put(20,0){\circle*{1.5}}
\put(20,10){\circle*{1.5}}
\put(20,20){\circle*{1.5}}
\put(20,30){\circle*{1.5}}
\put(30,0){\circle*{1.5}}
\put(30,10){\circle*{1.5}}
\put(30,20){\circle*{1.5}}
\put(30,30){\circle*{1.5}}
\put(40,0){\circle*{1.5}}
\put(40,10){\circle*{1.5}}
\put(40,20){\circle*{1.5}}
\put(40,30){\circle*{1.5}}

\put(60,13){$\D_2'$}
\put(70,0){\circle*{1.5}}
\put(70,10){\circle*{1.5}}
\put(70,20){\circle*{1.5}}
\put(70,30){\circle*{1.5}}
\put(80,0){\circle*{1.5}}
\put(80,10){\circle*{1.5}}
\put(80,20){\circle*{1.5}}
\put(80,30){\circle*{1.5}}
\put(90,0){\circle*{1.5}}
\put(90,10){\circle*{1.5}}
\put(90,20){\circle*{1.5}}
\put(90,30){\circle*{1.5}}
\put(100,0){\circle*{1.5}}
\put(100,10){\circle*{1.5}}
\put(100,20){\circle*{1.5}}
\put(100,30){\circle*{1.5}}
\put(110,0){\circle*{1.5}}
\put(110,10){\circle*{1.5}}
\put(110,20){\circle*{1.5}}
\put(110,30){\circle*{1.5}}

\put(20,0){\line(-1,1){20}}
\put(30,0){\line(-1,1){30}}
\put(40,0){\line(-1,1){30}}
\put(40,10){\line(-1,1){20}}

\put(80,0){\line(-1,1){10}}
\put(90,0){\line(-1,1){20}}
\put(100,0){\line(-1,1){30}}
\put(110,0){\line(-1,1){30}}
\put(110,10){\line(-1,1){20}}
\end{picture}
\end{center}
The initial form of $\D_2'$ is the square-free monomial
\[
x_{11}\ x_{12}\ x_{13}\ x_{21}\ x_{22}\ x_{23}\ x_{24}\ x_{31}\ x_{32}\ x_{33}\ x_{34}\ x_{35}\ x_{42}\ x_{43}\ x_{44}\ x_{45}\,.
\]
For arbitrary $m,n$, the initial form of $\D_0$ is the product of the $mn$ indeterminates.
\end{example}

\begin{proof}[Proof of Theorem~\ref{theorem:main}]
We first show that for each $k$ with $0\le k\le t-1$, one has 
\[
\fpt(I_t)\ \le\ \frac{(m-k)(n-k)}{t-k}\,.
\]

Let $\delta_k$ and $\delta_t$ be minors of size $k$ and $t$ respectively. Theorem~\ref{theorem:Bruns} implies that
\[
\delta_k^{t-k-1}\delta_t\ \in\ \bar{I_{k+1}^{t-k}}\,,
\]
and hence that $\delta_k^{t-k-1}I_t\ \subseteq\ \bar{I_{k+1}^{t-k}}$. By the Brian\c con-Skoda theorem, see, for example, \cite[Theorem~5.4]{HH:JAMS}, there exists an integer $N$ such that
\[
\left(\delta_k^{t-k-1}I_t\right)^{N+l}\ \in\ I_{k+1}^{(t-k)l}
\]
for each integer $l\ge1$. Localizing at the prime ideal $I_{k+1}$ of $R$, one has
\[
I_t^{N+l}\ \subseteq\ I_{k+1}^{(t-k)l}R_{I_{k+1}}\qquad\text{ for each }l\ge1\,,
\]
as the element $\delta_k$ is a unit in $R_{I_{k+1}}$. Since $R_{I_{k+1}}$ is a regular local ring of dimension $(m-k)(n-k)$, with maximal ideal $I_{k+1}R_{I_{k+1}}$, it follows that
\[
I_t^{N+l}\ \subseteq\ I_{k+1}^{[q]}R_{I_{k+1}}
\]
for positive integers $l$ and $q=p^e$ satisfying
\[
(t-k)l\ >\ (q-1)(m-k)(n-k)\,.
\]
Returning to the polynomial ring $R$, the ideal $I_{k+1}$ is the unique associated prime of $I_{k+1}^{[q]}$; this follows from the flatness of the Frobenius endomorphism, see for example, \cite[Corollary~21.11]{twentyfour}. Hence, in the ring $R$, we have
\[
I_t^{N+l}\ \subseteq\ I_{k+1}^{[q]}
\]
for all integers $q,l$ satisfying the above inequality. This implies that
\[
\nu_{I_t}(q)\ \le\ N+\frac{(q-1)(m-k)(n-k)}{t-k}\,.
\]
Dividing by $q$ and passing to the limit, one obtains
\[
\fpt(I_t)\ \le\ \frac{(m-k)(n-k)}{t-k}\,.
\]

Next, fix $k$ and $u$ be as in Lemma~\ref{lemma:min}, and consider $\D_k$ and $\D_k'$ as in Notation~\ref{notation}; the latter is defined only in the case $k\ge1$. Set
\[
\D=\begin{cases}
\D_0^t & \text{ if }\quad k=0,\\
\D_k^u\cdot(\D_k')^{t-k-u} & \text{ if }\quad k\ge1\text{ and }u\ge0,\\
(\D_k')^{t-k+u}\cdot\D_{k-1}^{-u} & \text{ if }\quad k\ge1\text{ and }u<0,
\end{cases}
\]
bearing in mind that $t-k-u\ge0$ by Lemma~\ref{lemma:min}.

We claim that $\D$ belongs to the integral closure of the ideal $I_t^{(m-k)(n-k)}$. This holds by Theorem~\ref{theorem:Bruns}, since, in each case,
\[
\deg\D=t(m-k)(n-k)\,,
\]
and $\D$ is a product of at most $(m-k)(n-k)$ minors: if $k\ge1$, then $\D$ is a product of exactly $(m-k)(n-k)$ minors, whereas if $k=0$ then $\D$ is a product of $t(m+n-1)$ minors and, by Lemma~\ref{lemma:min}, one has $t(m+n-1)\le mn$.

Let $\frakm$ be the homogeneous maximal ideal of $R$. For a positive integer $s$ that is not necessarily a power of $p$, set
\[
\frakm^{[s]}=(x_{ij}^s\mid i=1,\dots,m,\ j=1,\dots,n)\,.
\]
Using the lexicographical term order from Notation~\ref{notation}, the initial forms $\init(\D_k)$ and $\init(\D_k')$ are square-free monomials, and
\[
\init(\D)=\begin{cases}
\init(\D_0)^t & \text{ if }\quad k=0,\\
\init(\D_k)^u\cdot\init(\D_k')^{t-k-u} & \text{ if }\quad k\ge1\text{ and }u\ge0,\\
\init(\D_k')^{t-k+u}\cdot\init(\D_{k-1})^{-u} & \text{ if }\quad k\ge1\text{ and }u<0.
\end{cases}
\]
Thus, each variable $x_{ij}$ occurs in the monomial $\init(\D)$ with exponent at most $t-k$. It follows that
\[
\D\ \notin\ \frakm^{[t-k+1]}\,. 
\]

As $\D$ belongs to the integral closure of $I_t^{(m-k)(n-k)}$, there exists a nonzero homogeneous polynomial $f\in R$ such that
\[
f\D^l\ \in\ I_t^{(m-k)(n-k)l}\qquad\text{ for all integers }l\ge1\,.
\]
But then
\[
f\D^l\ \in\ I_t^{(m-k)(n-k)l}\setminus\frakm^{[q]}
\]
for all integers $l$ with $\deg f+l(t-k)\le q-1$. Hence,
\[
\nu_{I_t}(q)\ \ge\ (m-k)(n-k)l\qquad\text{ for all integers $l$ with }l\le\frac{q-1-\deg f}{t-k}\,.
\]
Thus,
\[
\nu_{I_t}(q)\ \ge\ (m-k)(n-k)\left(\frac{q-1-\deg f}{t-k}-1\right)\,,
\]
and dividing by $q$ and passing to the limit, one obtains
\[
\fpt(I_t)\ \ge\ \frac{(m-k)(n-k)}{t-k}\,,
\]
which completes the proof.
\end{proof}

%%%%%%%%%%%%%%%%%%%%%%%%%%%%%%%%%%%%%%%%%%%%%%%%%%%%%%%%%%%%%%%%%%%%%%%%

\end{document}